\documentclass[10pt]{elsarticle}

\usepackage{lineno,hyperref}
\modulolinenumbers[5]
\usepackage[all]{xy}
\usepackage{tabularx}

\usepackage{amsmath,amssymb}
\usepackage{amsthm}
\usepackage{latexsym}
\usepackage{cleveref}
\usepackage{enumerate}
\usepackage{bm}
\usepackage{multirow}
\usepackage{geometry}
\usepackage{indentfirst}
\usepackage{fancyhdr}
\usepackage{graphicx}
\usepackage{subfigure}
\usepackage[noend]{algpseudocode}
\usepackage{algorithmicx,algorithm}
\usepackage{threeparttable}
\usepackage{booktabs}
\usepackage{makecell}
\usepackage{color}
\usepackage{float}
\usepackage{url}

\newtheoremstyle{mystyle}
{3pt}
{3pt}
{\rm}
{}
{\bfseries}
{.}
{.5em}
{}
\newcommand{\RNum}[1]{\uppercase\expandafter{\romannumeral #1\relax}}

\newtheorem{lemma}{Lemma}

\newtheorem{corollary}{Corollary}[section]

\newtheorem{remark}{Remark}[section]
\newtheorem{theorem}{Theorem}
\theoremstyle{definition}
\newtheorem{example}{Example}

\begin{document}
\begin{frontmatter}
\title{Constructing Cycles in Isogeny Graphs of Supersingular Elliptic Curves}

\author[CAS,UCAS]{Guanju Xiao}
\ead{gjXiao@amss.ac.cn}

\author[CAS,UCAS]{Lixia Luo}
\ead{luolixia@amss.ac.cn}

\author[CAS,UCAS]{Yingpu Deng}
\ead{dengyp@amss.ac.cn}

\address[CAS]{Key Laboratory of Mathematics Mechanization, NCMIS, Academy of Mathematics and Systems Science, Chinese Academy of Sciences, Beijing 100190, People's Republic of China}
\address[UCAS]{School of Mathematical Sciences, University of Chinese Academy of Sciences, Beijing 100049, People's Republic of China}

\begin{abstract}
Loops and cycles play an important role in computing endomorphism rings of supersingular elliptic curves and related cryptosystems. For a supersingular elliptic curve $E$ defined over $\mathbb{F}_{p^2}$, if an imaginary quadratic order $O$ can be embedded in $\text{End}(E)$ and a prime $L$ splits into two principal ideals in $O$, we construct loops or cycles in the supersingular $L$-isogeny graph at the vertices which are next to $j(E)$ in the supersingular $\ell$-isogeny graph where $\ell$ is a prime different from $L$. Next, we discuss the lengths of these cycles especially for $j(E)=1728$ and $0$. Finally, we also determine an upper bound on primes $p$ for which there are unexpected $2$-cycles if $\ell$ doesn't split in $O$.

\end{abstract}

\end{frontmatter}
\section{Introduction}
Elliptic curves over finite fields play an important role in cryptography. A recent research area, called isogeny-based cryptography, studies cryptosystems whose security is based on the difficulty of finding a path in isogeny graphs of supersingular elliptic curves. Moreover, the only known quantum algorithm for this problem, due to Biasse, Jao and Sankar \cite{MR3296936}, has exponential complexity. Until now, the efficient algorithms in \cite{MR2695524,MR3451433} to compute endomorphism rings or isogenies between supersingular elliptic curves use the isogeny graph, which is a Ramanujan graph introduced in \cite{MR1027904}. These algorithms have exponential complexity.\par
Let $\mathbb{F}_p$ be a finite field of characteristic $p$ with $p > 3$, and let $\overline{\mathbb{F}}_p$ denote its algebraic closure. Let $\ell$ be a prime different from $p$. The supersingular isogeny graph $\mathcal{G}_{\ell}(\overline{\mathbb{F}}_p)$ is a directed graph whose vertices are the $\overline{\mathbb{F}}_{p}$-isomorphism classes of supersingular elliptic curves defined over $\overline{\mathbb{F}}_{p}$, and whose directed arcs represent $\ell$-isogenies (up to a certain equivalence) defined over $\overline{\mathbb{F}}_p$. We label the vertices of $\mathcal{G}_{\ell}(\overline{\mathbb{F}}_p)$ with their $j$-invariants.\par
Cryptographic applications based on the hardness of computing isogenies between supersingular elliptic curves were first proposed in 2006. Charles, Goren and Lauter constructed a hash function in \cite{MR2496385} from the supersingular isogeny graph $\mathcal{G}_{\ell}(\overline{\mathbb{F}}_{p})$. Finding collisions for the CGL hash function is connected to finding loops or cycles in the supersingular isogeny graph $\mathcal{G}_{\ell}(\overline{\mathbb{F}}_{p})$.\par
In 2011, Jao and De Feo \cite{MR2931459} (see also \cite{MR3259113}) presented a key agreement scheme whose security is based on the hardness of finding paths in the isogeny graph $\mathcal{G}_{\ell}(\overline{\mathbb{F}}_p)$ for small $\ell$ (typically $\ell = 2, 3$). There is also a submission \cite{2019} to the NIST PQC standardization competition based on supersingular isogeny problems. Moreover, Eisentr\"{a}ger et al. \cite{MR3794837} proved that finding paths in the supersingular $\ell$-isogeny graph is equivalent to computing the endomorphism rings of supersingular elliptic curves. Constructing cycles in supersingular isogeny graphs is important in algorithmic number theory and cryptography.\par
Adj et al. \cite{MR3872948} defined the supersingular isogeny graph $\mathcal{G}_{\ell} (\mathbb{F}_{p^2})$ whose vertices are (representatives of) the $\mathbb{F}_{p^2}$-isomorphism classes of supersingular elliptic curves defined over $\mathbb{F}_{p^2}$, and whose directed arcs represent degree-$\ell$ $\mathbb{F}_{p^2}$-isogenies between the elliptic curves. Adj et al. \cite{MR3872948} described clearly the three subgraphs of $\mathcal{G}_{\ell}(\mathbb{F}_{p^2})$ whose vertices correspond to supersingular elliptic curves $E$ over $\mathbb{F}_{p^2}$ with $t=p^2+1-\#E(\mathbb{F}_{p^2}) \in \{ 0,-p,p \}$, and they also proved the following result:
$$\mathcal{G}_{\ell}(\overline{\mathbb{F}}_{p}) \cong \mathcal{G}_{\ell}(\mathbb{F}_{p^2},2p) \cong \mathcal{G}_{\ell}(\mathbb{F}_{p^2},-2p).$$
Moreover, Adj et al. and Ouyang-Xu \cite{MR3947814} proved the following results about the loops at the vertices $1728$ and $0$ in $\mathcal{G}_{\ell}(\mathbb{F}_{p^2},2p)$. For $\ell>3$ a prime integer, if $p\equiv 3\ \text{mod 4}$ and $p>4 \ell$, there are either 2 or 0 loops at 1728 if $\ell\equiv  1 \ \text{mod } 4$ or $3 \ \text{mod} \ 4$ respectively; and if $p\equiv 2\ \text{mod} \ 3$ and $p>3 \ell$, there are either 2 or 0 loops at 0 if $\ell\equiv 1 \ \text{mod} \ 3$ or $2 \ \text{mod} \ 3$ respectively. Li-Ouyang-Xu \cite{MR4019136} also described the neighborhood of vertices $1728$ and $0$ in $\mathcal{G}_{\ell}(\overline{\mathbb{F}}_{p})$.\par
The methods in \cite{MR3947814} and \cite{MR4019136} are based on the knowledge of the endomorphism rings of the supersingular elliptic curves corresponding to the vertices $0$ and $1728$. For a general supersingular elliptic curve $E$, it is very difficult to compute the endomorphism ring $\text{End}(E)$, but we may know a non-trivial endomorphism of $E$. Assume an imaginary quadratic order $O$ can be embedded in the endomorphism ring of $E$, we construct loops or cycles in the supersingular $L$-isogeny graph if a prime $L$ splits into two principal ideals in $O$. We also discuss the lengths of these cycles. Since the results for $j=1728$ and $0$ are more explicit, we will discuss these special cases separately. For a prime $p$, the vertices in different supersingular isogeny graphs are the same, and our results show a deeper connection between these supersingular isogeny graphs. In this paper, a $m$-cycle means a simple cycle (as defined in \cite{MR2695524}) with $m$ vertices and a loop is a $1$-cycle. We will denote by $\ell$ and $L$ two different primes.\par
The remainder of this paper is organized as follows. In Section 2, we provide preliminaries on elliptic curves over finite fields, maximal orders of $B_{p,\infty}$ and modular polynomials. We construct loops and cycles in Section 3 and discuss the lengths of these cycles in Section 4. In Section 5, we determine an upper bound on primes $p$ for which there are unexpected $2$-cycles in $\mathcal{G}_{L}(\overline{\mathbb{F}}_{p})$. Finally, we make a conclusion in Section 6.\par

\section{Preliminaries}

\subsection{Elliptic Curves over Finite Fields}
We recall basic facts about elliptic curves over finite fields. The general references are \cite{MR2514094,MR2404461}. In the remainder of this paper, $p$ and $\ell$ will denote different prime integers with $p>3$.\par
Let $\mathbb{F}_q$ be a finite extension of $\mathbb{F}_p$ and $\overline{\mathbb{F}}_p$ be the algebraic closure of $\mathbb{F}_p$. An elliptic curve $E$ over a finite field $\mathbb{F}_q$ is defined by a Weierstrass equation $Y^2=X^3+aX+b$ where $a,b\in \mathbb{F}_q$ and $4a^3+27b^2\neq 0$. The chord-and-tangent addition law makes of $E(\mathbb{F}_q)=\left\{(x,y)\in \mathbb{F}_q ^2:y^2=x^3+ax+b\right\} \cup \left\{\infty \right\}$ an abelian group, where $\infty$ is the point at infinity. For any integer $n\geq 2$ with $p\nmid n$, the group of $n$-torsion points on $E$ is isomorphic to $\mathbb{Z}_n \oplus \mathbb{Z}_n$. In particular, if $n$ is prime then $E$ has exactly $n+1$ distinct subgroups of order $n$.\par
Let $E_1$ and $E_2$ be elliptic curves defined over $\mathbb{F}_q$. An isogeny from $E_1$ to $E_2$ is a morphism $\phi:E_1 \rightarrow E_2$ satisfying $\phi(\infty)=\infty$. In this paper, the isogenies are always nonconstant. An isogeny $\phi$ is a surjective group homomorphism with finite kernel. Every $\mathbb{F}_q$-isogeny can be represented as $\phi = (r_1(X), r_2(X) \cdot Y)$ where $r_1$, $r_2 \in \mathbb{F}_q (X)$. Let $r_1(X)=p_1(X)/q_1(X)$, where $p_1, q_1 \in \mathbb{F}_q[X]$ with $\text{gcd}(p_1,q_1)=1$. The degree of $\phi$ is $\text{max}(\text{deg}\ p_1, \text{deg}\ q_1)$ and $\phi$ is said to be separable if $r_1 '(X) \neq 0$. Note that all isogenies of prime degree $\ell\neq p$ are separable. If $\phi:E_1 \rightarrow E_2$ is an isogeny of degree $m$, then there exists a unique isogeny $\hat{\phi}:E_2 \rightarrow E_1$ satisfying $\hat{\phi} \circ \phi=[m]$ and $\phi \circ \hat{\phi}=[m]$, where $[m]$ is the multiplication-by-$m$ map with degree $m^2$. We call $\hat{\phi}$ the dual of $\phi$. The following lemma is in chapter 3 of \cite{MR2514094}.
\begin{lemma}
	Let $\phi:E_1\rightarrow E_2$ and $\psi:E_1\rightarrow E_3$ be nonconstant isogenies, and assume that $\phi$ is separable. If $\text{ker}(\phi) \subseteq \text{ker} (\psi)$, then there is a unique isogeny $\lambda :E_2 \rightarrow E_3$ satisfying $\psi =\lambda \circ \phi$.
\end{lemma}
An endomorphism of $E$ is an isogeny from $E$ to itself. The Frobenius map $\pi :(x,y) \mapsto (x^q,y^q)$ is an inseparable endomorphism. The characteristic polynomial of $\pi$ is $x^2-tx+q$, where $t$ is the trace of $\pi$ and the Hasse's Theorem implies that $\vert t \vert \leq 2 \sqrt{q}$ and $\# E(\mathbb{F}_q)=q+1-t$. Tate's Theorem asserts that $E_1$ and $E_2$ are $\mathbb{F}_q$-isogenous if and only if  $\#E_1(\mathbb{F}_q)=\# E_2(\mathbb{F}_q)$. It is well known that $E$ is supersingular (resp. ordinary) if and only if $p\mid t$ (resp. $p \nmid t$).\par
The $j$-invariant of $E$ is $j(E)=1728\cdot 4a^3/(4a^3+27b^2)$. One can easily check that $j(E)=0$ if and only if $a=0$, and  $j(E)=1728$ if and only if $b=0$.
Different elliptic curves with the same $j$-invariant are isomorphic over the algebraic closure $\overline{\mathbb{F}}_p$.
An automorphism of $E$ is an isomorphism from $E$ to itself. The group of all automorphisms of $E$ that are defined over $\overline{\mathbb{F}}_p$ is denoted by $\text{Aut}(E)$. As we know (see \cite[Chapter 3.10]{MR2514094}), $\text{Aut}(E)\cong \{\pm 1\}$ if $j(E)\neq 0,1728$. If $j(E)=1728$, then $\text{Aut}(E)$ is a cyclic group of order 4 with generator $\theta:(x,y)\mapsto (-x,iy)$ where $i$ is a primitive fourth root of unity. If $j(E)=0$, then $\text{Aut}(E)$ is a cyclic group of order 6 with generator $\omega:(x,y)\mapsto (\eta x,-y)$ where $\eta$ is a primitive third root of unity.\par
Moreover, the $j$-invariant of any supersingular elliptic curve over $\overline{\mathbb{F}}_p$ is proved to be in $\mathbb{F}_{p^2}$ \cite{MR2514094} and it is called a supersingular $j$-invariant. From now on, we suppose that $E$ is supersingular, and, since $j(E)\in \mathbb{F}_{p^2}$, we assume that $E$ is defined over $\mathbb{F}_{p^2}$. Schoof \cite{MR914657} determined the number of isomorphism classes of elliptic curves over a finite field. The number of supersingular $j$-invariants is $[\frac{p}{12}]+\epsilon$, where $\epsilon =0,\ 1,\ 1,\ 2$ if $p\equiv 1,\ 5,\ 7,\ 11 \ (\text{mod}\ 12)$ respectively.\par
The supersingular isogeny graph $\mathcal{G}_{\ell}(\overline{\mathbb{F}}_p)$ is a Ramanujan graph (see \cite{MR2496385}) whose vertices are the supersingular $j$-invariants and edges are equivalent classes of $\ell$-isogenies defined over $\overline{\mathbb{F}}_p$. Let $\phi_1,\phi_2:E(j_1)\to E(j_2)$ be two $\ell$-isogenies defined over $\overline{\mathbb{F}}_p$. We say that $\phi_1$ and $\phi_2$ are equivalent if they have the same kernel, or equivalently, if there exist $\rho_2\in \text{Aut}(E(j_2))$ such that $\phi_2=\rho_2 \circ \phi_1$. Adj et al. proved $\mathcal{G}_{\ell}(\mathbb{F}_{p^2},2p) \cong \mathcal{G}_{\ell}(\mathbb{F}_{p^2},-2p) \cong \mathcal{G}_{\ell}(\overline{\mathbb{F}}_{p})$. In the remainder of this paper, we will use the symbol $\mathcal{G}_{\ell}(\overline{\mathbb{F}}_{p})$.\par

\subsection{Endomorphism Ring and Quaternion Algebra}
If $E$ is supersingular elliptic curve, the endomorphism ring $\text{End}(E)$ is a maximal order of $B_{p,\infty}$ where $B_{p,\infty}$ is a quaternion algebra \cite{MR580949,MR2869057.pdf} defined over $\mathbb{Q}$ and ramified at $p$ and $\infty$. The reduced trace Trd and the reduced norm Nrd of $\alpha \in B_{p,\infty}$ are defined as:
$$\text{Trd}(\alpha)=\alpha + \bar{\alpha}, \quad \text{Nrd}(\alpha)=\alpha \bar{\alpha}$$
where $\bar{\alpha}$ is the canonical involution of $\alpha$.\par
An order $\mathcal{O}$ of $B_{p,\infty}$ is a subring of $B_{p,\infty}$ which is also a lattice, and it is called a maximal order if it is not properly contained in any other order. Two orders $\mathcal{O}_1$ and $\mathcal{O}_2$ are equivalent if and only if there exists $\alpha \in B_{p,\infty}^{*}$ such that $\mathcal{O}_1=\alpha^{-1} \mathcal{O}_2 \alpha$.
For $\mathcal{O}$ a maximal order of $B_{p,\infty}$, let $I$ be a left ideal of $\mathcal{O}$. Define the left order $\mathcal{O}_L(I)$ and the right order $\mathcal{O}_R(I)$ of $I$ by
$$\mathcal{O}_L(I)=\left\{x\in B_{p,\infty}:xI\subseteq I \right \}, \quad \mathcal{O}_R(I)=\left\{x\in B_{p,\infty}:Ix \subseteq I \right \}$$
Moreover, $\mathcal{O}_L(I)=\mathcal{O}$ and $\mathcal{O}_R(I)=\mathcal{O}^{'}$ is also a maximal order, in which case we say that $I$ connects $\mathcal{O}$ and $\mathcal{O}^{'}$. Moreover, if $\mathcal{O}$ is maximal, then $\mathcal{O}_R(I) = \mathcal{O}$ if and only if $I$ is principal. The reduced norm of $I$ can be defined as $$\text{Nrd}(I)=\text{gcd}(\{\text{Nrd}(\alpha) | \alpha \in I \}).$$
Fix a maximal order $\mathcal{O}$. Any left ideal of $\mathcal{O}$ with reduced norm $\ell$ can be written as $I=\mathcal{O} \ell+\mathcal{O} \alpha$ where $\alpha \in \mathcal{O}$ is such that $\ell \mid \text{Nrd}(\alpha)$. For any $I_1$, $I_2$ left ideals of $\mathcal{O}$, $I_1$ and $I_2$ belong to the same ideal class if and only if there exists $\mu \in B_{p,\infty}^{*}$ such that $I_1=I_2 \mu$. Moreover, if $\text{Nrd}(I_1)=\text{Nrd}(I_2)$, then $\text{Nrd}(\mu)=1$. Let $X_{\ell}$ be the set of all left $\mathcal{O}$-ideals of reduced norm $\ell$, there are $\ell+1$ ideals in $X_{\ell}$. Given a quadratic order $O$ and a maximal order $\mathcal{O}$ of $B_{p, \infty}$, we say that $O$ is optimally embedded in $\mathcal{O}$ if $O=\mathcal{O}\cap K$ for some subfield $K\subseteq B_{p,\infty}$.\par
A theorem by Deuring \cite{MR5125} gives an equivalence of categories between the supersingular $j$-invariants and the maximal orders in the quaternion algebra $B_{p,\infty}$. Furthermore, if $E$ is
 an elliptic curve with $\text{End}(E)=\mathcal{O}$, there is a one-to-one correspondence between isogenies $\phi:E\to E^{'}$ and left $\mathcal{O}$-ideals $I$. More details on the correspondence can be found in Chapter 42 of \cite{MR2869057.pdf}.

\subsection{$j$-Function and Modular Polynomials}
In this subsection, we present some properties of the $j$-function. The reader can refer to \cite{MR3236783, MR2404461} for more details.
Given $\tau$ in the upper half plane $\mathcal{H}$, we get a lattice $[1,\tau]$ and the $j$-function $j(\tau)$ is defined by
$$j(\tau)=j([1,\tau])=1728 \frac{(1+240\sum_{k=1}^{\infty}\frac{k^3q^k}{1-q^k})^3}{(1+240\sum_{k=1}^{\infty}\frac{k^3q^k}{1-q^k})^3-(1-504\sum_{k=1}^{\infty}\frac{k^5q^k}{1-q^k})^2}.$$\par
Let $K$ be an imaginary quadratic field. If $\tau \in K \setminus \mathbb{Q}$, then $L=[1,\tau]$ is a lattice in $K$. We can define the order $O$ of $L$ to be the set of elements $\lambda \in K$ such that $\lambda L \subseteq L$. It is well known that the elliptic curve $E(j(\tau))$ defined over $\mathbb{C}$ with $j$-invariant $j(\tau)$ has complex multiplication by $O$. Cox lists the 13 orders with class number one and the corresponding $j$-invariants in \S 12 of \cite{MR3236783}.\par
Deuring's reducing and lifting theorems in \cite{MR5125} describe the structures of endomorphism rings which are preserved in passing between elliptic curves over every field.\par
For any $\tau \in \mathcal{H}$, the complex numbers $j(\tau)$ and $j(N\tau)$ are the $j$-invariants of elliptic curves defined over $\mathbb{C}$ that are related by an isogeny whose kernel is a cyclic group of order $N$. The minimal polynomial $\Phi_N(Y)$ of the function $j(Nz)$ over the field $\mathbb{C}(j(z))$ has coefficients that are integer polynomials in $j(z)$. If we replace $j(z)$ with $X$, we obtain the modular polynomial $\Phi_N \in \mathbb{Z}[X,Y]$ which is symmetric in $X$ and $Y$ and has degree $N \prod_{\ell\mid N} (1+\frac{1}{\ell})$ in both variables.\par
When $N$ is a prime integer, every $N$-isogeny is cyclic, and we have
\begin{center}
  $\Phi_{N}(j(E_1),j(E_2))=0  \Longleftrightarrow E_1$ and $E_2$ are $N$-isogenous.
\end{center}
This moduli interpretation remains valid over every field, even those of positive characteristic.

\section{Constructing Cycles}
In this section, we will construct loops or cycles at some vertices in supersingular isogeny graphs. We assume that $L$ and $\ell$ are different primes.\par
Let $E$ be a supersingular elliptic curve defined over $\mathbb{F}_{p^2}$, and assume that an imaginary quadratic order $\mathbb{Z}[\tau]$ can be optimally embedded in $\mathcal{O}\cong \text{End}(E)$. Suppose $E[\ell]=\langle P,Q\rangle$, where $P$ and $Q$ are two linearly independent $\ell$-torsion points. Recall that there are $\ell +1$ subgroups of $E[\ell]$ with order $\ell$. If $G_n$ is one of these $\ell +1$ subgroups, then $\phi_n: E \to E_n (\cong E/G_n)$ is an $\ell$-isogeny with kernel $G_n$. Let $\mathcal{G}_L(\overline{\mathbb{F}}_p,\ell+1)$ denote the subgraph of $\mathcal{G}_L(\overline{\mathbb{F}}_p)$ which consists of $j(E_n)$ for $n=0,\ldots,\ell$. We have the following theorem.
\begin{theorem}
If an imaginary quadratic order $\mathbb{Z}[\tau]$ is optimally embedded in $\mathcal{O} \cong \text{End}(E)$, then
there are loops or cycles at $j(E_n)$ in $\mathcal{G}_L(\overline{\mathbb{F}}_p,\ell+1)$ for every $n\in \{0,\ldots,\ell \}$ where $L$ splits into two principal ideals in $\mathbb{Z}[\tau]$.
\end{theorem}
\begin{proof}
We assume $\tau=\sqrt{-d}$ where $d$ is a positive integer, and the other case $\tau=\frac{1+\sqrt{-d}}{2}$ can be proved similarly.\par
If $\mathbb{Z}[\sqrt{-d}]$ is optimally embedded in $\text{End}(E)$, then $\mathbb{Z}[\ell \sqrt{-d}]$ can be embedded in $\mathcal{O}_n$ where $\mathcal{O}_n$ is the endomorphism ring of $E_n$. Since $L$ splits into two principal ideals in $\mathbb{Z}[\sqrt{-d}]$, we can write $L=a^2+db^2$ with $a,b \in \mathbb{Z}$.\par
If $\ell \mid b$, then $L$ can be written as $L=\alpha \bar{\alpha}$ with $\alpha, \ \bar{\alpha}\in \mathbb{Z}[\ell \sqrt{-d}]$. In this case, there are at least 2 loops at $j(E_n)$ for $n=0,\ldots,\ell$  in $\mathcal{G}_L(\overline{\mathbb{F}}_p,\ell+1)$.\par
If $\ell\nmid b$, then $[a\pm b\sqrt{-d}](E[\ell]) = E[\ell]$ since $\text{deg}([a\pm b\sqrt{-d}])=L\neq \ell$. In other words, $[a\pm b\sqrt{-d}]$ is a bijection on $E[\ell]$, so that $[a\pm b\sqrt{-d}]$ acts as a permutation on the set of subgroups $\{ G_i \}_{i=0,\ldots,\ell}$. Considering the following two isogenies:
   $$\Psi_{n,\pm}:\xymatrix{&E_n \ar[r]^{\hat{\phi}_n} &E \ar[r]^{[a\pm b\sqrt{-d}]} &E}$$
where $\hat{\phi}_n$ is the dual isogeny of $\phi_n$, we have $\Psi_{n,+}(E_n[\ell])=[a + b\sqrt{-d}](G_n)$ (resp. $\Psi_{n,-}(E_n[\ell])=[a - b\sqrt{-d}](G_n)$) is the kernel of some $\phi_{n_1}$ (resp. $\phi_{n_2}$). The isogenies $\phi_{n_1}\circ \Psi_{n,+}:E_n \to E_{n_1}$ and $\phi_{n_2}\circ \Psi_{n,-}:E_n \to E_{n_2}$ factor through $[\ell]\in \text{End}(E_n)$ by Lemma 1, so there are two $L$-isogenies $\psi_{n,+}:E_n \to E_{n_1}$ and $\psi_{n,-}:E_n \to E_{n_2}$ such that $\phi_{n_1}\circ \Psi_{n,+}=\psi_{n,+} \circ [\ell]$ and $\phi_{n_2}\circ \Psi_{n,-}=\psi_{n,-} \circ [\ell]$. Since $\text{ker} (\psi_{n,+}) \subseteq \text{ker}([L]) \cap \text{ker}([\ell(a+b\sqrt{-d})])$ and $\text{gcd}(\text{deg}([L]),\text{deg}([\ell(a+b\sqrt{-d})]))=L$, the kernel ideal of $\psi_{n,+}$ is $I_{n,1}=(L,\ell(a+b\sqrt{-d}))$. Similarly, the kernel ideal of $\psi_{n,-}$ is $I_{n,2}=(L,\ell(a-b\sqrt{-d}))$. \par
If $E_{n_1}$ is isomorphic to $E_n$. Therefore $I_{n,1}$ is a principal left ideal of $\mathcal{O}_n$ and $\psi_{n,+}$ is an endomorphism of $E_n$. Because $I_{n,2}$ is the conjugate ideal of $I_{n,1}$, $I_{n,2}$ is also principal and $\psi_{n,-}$ is an endomorphism of $E_n$. In this case, $j(E_{n_2})=j(E_{n_1})=j(E_n)$ and we construct two loops at $j(E_n)$ in $\mathcal{G}_L(\overline{\mathbb{F}}_p,\ell+1)$.\par
If $E_{n_1}$ is not isomorphic to $E_n$ and $E_{n_1}$ is isomorphic to $E_{n_2}$, then $j(E_{n_1})=j(E_{n_2})$ and $\psi_{n,+}$ and $\psi_{n,-}$ are two different $L$-isogenies since $I_{n,1} \neq I_{n,2}$. Therefore $\hat{\psi}_{n,-}\circ \psi_{n,+}$ and $\hat{\psi}_{n,+}\circ \psi_{n,-}$ are two $2$-cycles at $j(E_n)$ in $\mathcal{G}_L(\overline{\mathbb{F}}_p,\ell+1)$. Furthermore, we have $\mathcal{O}_{n_1}=\mathcal{O}_{n_2}$ and the right order of $I_{n,1}$ and $I_{n,2}$ is $\mathcal{O}_{n_1}$, so the right order of $I_{n_1,1}$ and $I_{n_1,2}$ is $\mathcal{O}_{n}$. Since the corresponding isogenies of $I_{n_1,1}$ and $I_{n_1,2}$ are $\psi_{n_1,+}$ and $\psi_{n_1,-}$, we also construct two $2$-cycles at $j(E_{n_1})$ in $\mathcal{G}_L(\overline{\mathbb{F}}_p,\ell+1)$.\par
If $E_n$, $E_{n_1}$ and $E_{n_2}$ are not isomorphic, we denote the target elliptic curve of $\psi_{n_1,+}$ by $E_{n_3}$. We have that $E_{n_3}$ is not isomorphic to $E_n$ or $E_{n_1}$, otherwise there exists a contradiction with the above two cases. If $E_{n_3}$ is isomorphic to $E_{n_2}$, then $\psi_{n_2,+} \circ \psi_{n_1,+} \circ \psi_{n,+}$ is a cycle through $j(E_n)$, $j(E_{n_1})$ and $j(E_{n_2})$ in $\mathcal{G}_L(\overline{\mathbb{F}}_p,\ell+1)$. If $E_{n_3}$ is not isomorphic to $E_{n_2}$, we denote the target elliptic curves of $\psi_{n_3,+}$ by $E_{n_4}$. We have that $E_{n_4}$ is not isomorphic to $E_{n_1}$ or $E_{n_3}$, Moreover, we claim that $E_{n_4}$ is not isomorphic to $E_n$. If $E_{n_4}$ is isomorphic to $E_n$, then the dual of $\psi_{n_3,+}$ is $\psi_{n,-}$ which is the dual of $\psi_{n_2,+}$. This means that $E_{n_3}$ is isomorphic to $E_{n_2}$, so we get a contradiction. If $E_{n_4}$ is isomorphic to $E_{n_2}$, then we construct a cycle through $j(E_n)$, $j(E_{n_1})$, $j(E_{n_3})$ and $j(E_{n_2})$ in $\mathcal{G}_L(\overline{\mathbb{F}}_p,\ell+1)$. The following process is similar. Since there are at most $\ell+1$ vertices, we construct cycles at $j_n$ in $\mathcal{G}_L(\overline{\mathbb{F}}_p,\ell+1)$.\par
\end{proof}
\begin{remark}
  If $\mathbb{Z}[\sqrt{-d}]$ can be embedded in the endomorphism ring of $E_n$, then $I_{n,1}$ and $I_{n,2}$ are principal and there exist two loops at $j(E_n)$ in $\mathcal{G}_L(\overline{\mathbb{F}}_p,\ell+1)$.
\end{remark}
The following example illustrates Theorem 1.
\begin{example}
Let $p=3461$ and $\ell=5$. Since $\big( \frac{-7}{3461} \big)=-1$, $j(\sqrt{-7})=255^3 \equiv 3185 \bmod {3461} $ is a supersingular $j$-invariant in $\mathbb{F}_p$. Let $\mathbb{F}_{p^2}=\mathbb{F}_p(\beta)$ where $\beta^2+\beta+1=0$ in $\mathbb{F}_{p^2}$. The solutions of $\Phi_5(X,3185)$ are $j_0=819$, $j_1=2402$, $j_2=2591\beta+1415$, $j_3=1039\beta+2586$, $j_4=870\beta+2285$ and $j_5=2422\beta+1547$ in $\mathbb{F}_{p^2}$.
By computing modular polynomials \cite{MR2869057.org}, we get the subgraph $\mathcal{G}_L(\overline{\mathbb{F}}_p,\ell+1)$ which consists of these 6 vertices.\par
For $L=11=2^2+7\cdot 1^2$, we have the following graph $\mathcal{G}_{11}(\overline{\mathbb{F}}_{3461},6)$.
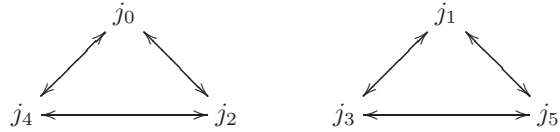
\[\xymatrix{
&&j_0\ar[dl] \ar[dr]&\\
&j_4 \ar[rr] \ar[ur] &&j_2 \ar[ll] \ar[ul]\\}
\xymatrix{
&&j_1\ar[dl] \ar[dr]&\\
&j_3 \ar[rr] \ar[ur] &&j_5 \ar[ll] \ar[ul]\\}\]\par
For $L=23=4^2+7\cdot 1^2$, we have the following graph $\mathcal{G}_{23}(\overline{\mathbb{F}}_{3461},6)$.
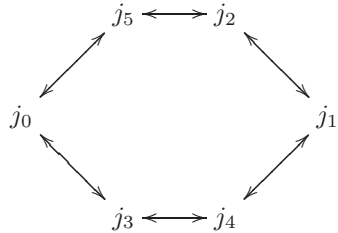
\[\xymatrix{
&&j_5 \ar[r] \ar[dl] &j_2 \ar[l] \ar[dr] &\\
&j_0 \ar[ur] \ar[dr] &&&j_1 \ar[ul] \ar[dl]\\
&&j_3 \ar[r] \ar[ul] &j_4 \ar[l] \ar[ur] & \\ }\]
\end{example}

Next, we will discuss the usefulness of Theorem 1 in CGL hash function and the imaginary quadratic order $O$ which can be embedded in $\mathcal{O}\cong \text{End}(E)$.
The following example implies that we can find cycles in supersingular isogeny graphs by Theorem 1.
\begin{example}
  Let $p=12601\equiv 6 \ \text{mod} \ 11$, we have that $j(\frac{1+\sqrt{-11}}{2})=-32^3 \equiv 5035$ is a supersingular $j$-invariant in $\mathbb{F}_p$. Moreover $4825$ is a root of $H_{-44}(x)$ in $\mathbb{F}_p$ and $j_0=5035$, $j_1=7022\beta+1350$ and $j_2=5579\beta+1350$ are the vertices next to $4825$ in $\mathcal{G}_{2}(\overline{\mathbb{F}}_p)$ where $\beta^2+11=0$ in $\mathbb{F}_{p^2}$. We have the following subgraph of $\mathcal{G}_{47}(\overline{\mathbb{F}}_p)$ which consists of $j_0$, $j_1$ and $j_2$.
  
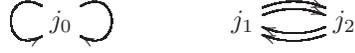
\[\xymatrix{ j_0 \ar@(ul,dl) \ar@(ur,dr)
 && j_1\ar@/^/@<-0.5mm>[r] \ar@/^/@<0.5mm>[r] &j_2\ar@/^/@<-0.5mm>[l] \ar@/^/@<0.5mm>[l] }\]
 Charles, Goren and Lauter proved in \cite{MR2496385} that if $\mathcal{G}_{2}(\overline{\mathbb{F}}_p)$ has no $2$-cycles then $p\equiv 1 \ \text{mod} \ 840$. In this example, we construct $2$-cycles in $\mathcal{G}_{47}(\overline{\mathbb{F}}_p)$ for $p=12601 \equiv 1\ \text{mod} \ 840$.
\end{example}
Based on Theorem 1 and  these examples, we can say more about the collision resistance of the hash function defined in \cite{MR2496385}. In addition to choosing an appropriate prime $p$ as discussed in \cite{MR2496385}, the prime $\ell$ cannot split in imaginary quadratic orders which can be embedded in the endomorphism rings.\par
The order $\mathbb{Z}[\tau]$ plays an important role in Theorem 1. What can we say about the discriminant of $\mathbb{Z}[\tau]$? Kaneko \cite{Kan} proved that the endomorphism ring of any supersingular elliptic curve defined over $\mathbb{F}_p$ contains an imaginary quadratic order $O_{-D}$ with discriminant $-D$ satisfying $D\le \frac{4}{\sqrt{3}} \sqrt{p}$. Recently, Love and Boneh \cite{http20200228} proved that the endomorphism ring of any supersingular elliptic curve contains an imaginary quadratic order $O_{-D}$ with $D< 2p^{\frac{2}{3}}+1$.

\section{Lengths of These Cycles}
We will discuss the lengths of the cycles which we construct in Section 3. We also assume that $\mathbb{Z}[\tau]$ is optimally embedded in $\mathcal{O} \cong \text{End}(E)$. As in the proof of Theorem 1, $\psi_{n,+}:E_n \to E_{n_1}$ and $\psi_{n,-}:E_n \to E_{n_2}$ are two $L$-isogenies. If $\mathbb{Z}[\tau]$ can be embedded in $\mathcal{O}_n \cong \text{End}(E_n)$, then $I_{n,1}$ and $I_{n,2}$ are principal and we construct two loops at $j(E_n)$ in $\mathcal{G}_L(\overline{\mathbb{F}}_p,\ell+1)$. If $L\mathbb{Z}[\ell \tau]=\mathfrak{L}\mathfrak{L}'$, we define $m$ to be the order of $\mathfrak{L}$ in the class group of $\mathbb{Z}[\ell \tau]$. Let $D$ be the absolute value of the discriminant of $\mathbb{Z}[\tau]$, then $D=4d$ if $\tau=\sqrt{-d}$ and $D=d$ if $\tau=\frac{1+\sqrt{-d}}{2}$ with $d \equiv \ 3 \ \text{mod} \ 4$.
\begin{theorem}
  Suppose $j(E)\neq 0,1728$. Assume that $\mathbb{Z}[\tau]$ and $\mathbb{Z}[\ell \tau]$ are optimally embedded in $\mathcal{O}$ and $\mathcal{O}_n$ respectively where $E$ and $E_n$ are $\ell$-isogenous. If $p>\ell^2 L D$ and $\ell$ does not split in $\mathbb{Z}[\tau]$, then there exist two $m$-cycles at $j(E_n)$ in $\mathcal{G}_L(\overline{\mathbb{F}}_p,\ell+1)$ where $L$ splits into two principal ideals in $\mathbb{Z}[\tau]$.
\end{theorem}
\begin{proof}
Since $j(E)\neq 0,1728$, we have $\mathbb{Z}[\tau]\neq \mathbb{Z}[i], \mathbb{Z}[\frac{1+\sqrt{-3}}{2}]$ and the unit group of $\mathbb{Z}[\tau]$ is $\{\pm 1\}$. If $\tau=\sqrt{-d}$ and write $L=(a+b\sqrt{-d})(a-b\sqrt{-d})$, then $m$ is the smallest positive integer such that $(a+b\sqrt{-d})^m=x+y\sqrt{-d}$ with $\ell \mid y$.\par
If $\ell \mid b$, then $m=1$ and we construct two loops at $j(E_n)$ in $\mathcal{G}_L(\overline{\mathbb{F}}_p,\ell+1)$. We assume $\ell \nmid b$ in the following.\par
We claim that $E_n$ is not isomorphic to $E_{n_1}$ or $E_{n_2}$. If $E_n$ is isomorphic to $E_{n_1}$, then the right order of $I_{n,1}=(L,\ell(a+b\sqrt{-d}))$ is $\mathcal{O}_n$ and $I_{n,1}$ is principal. There exists an element $\alpha \in \mathcal{O}_n$ such that $\text{Nrd}(\alpha)=L$. We assume that $\mathbb{Z}[\ell \tau]$ is optimally embedded in $\mathcal{O}_n$ and $\ell \nmid b$, so $\alpha \notin \mathbb{Z}[\ell \tau]$. The absolute value $D'$ of the discriminant of $\mathbb{Z}[\alpha]$ satisfies $D'\leq 4 \text{Nrd}(\alpha)=4L$. By Theorem 2 in \cite{Kan}, we have $4p<4 \ell^2DL$ since two different imaginary orders $\mathbb{Z}[\ell \tau]$ and $\mathbb{Z}[\alpha]$ can be embedded in $\mathcal{O}_n$. If $p>\ell^2 L D$, then such $\alpha$ does not exist and $I_{n,1}$ is not principal.\par
If $m=2$, then $[a \pm b\sqrt{-d}]^2(G_n)=G_n$. Since $j(E_n)\neq j(E_{n_1})$, $\psi_{n_1,+}\circ
 \psi_{n,+}:E_n\to E_{n_1} \to E_n$ is a $2$-cycle at $j(E_n)$ in $\mathcal{G}_L(\overline{\mathbb{F}}_p,\ell+1)$. In this case, $\psi_{n,-}=\hat{\psi}_{n_1,+}$ and $j(E_{n_1})=j(E_{n_2})$.\par
If $m>2$, then $[a+b\sqrt{-d}]^m(G_n)=[x+y\sqrt{-d}](G_n)=G_n$ and $j(E_{n_1})\neq j(E_{n_2})$. The composition of the following $3m$ isogenies
  $$\xymatrix{&E_n \ar[r]^{\hat{\phi}_n} &E \ar[r]^{[a+ b\sqrt{-d}]} &E \ar[r]^{\phi_{n_1}} &E_{n_1} \ar[r]^{\hat{\phi}_{n_1}} &E \ar[r]^{[a+ b\sqrt{-d}]} &\ldots \ar[r] &E \ar[r]^{[a+ b\sqrt{-d}]} &E \ar[r]^{\phi_{n}} &E_n}$$
factors through $[\ell^m]$.\par
As in the proof of Theorem 1, the isogenies $E_n \rightarrow E \rightarrow E \rightarrow E_{n_1}$ factor through $[\ell] \in \text{End}(E_n)$ and we get an $L$-isogeny $\psi_{n,+}:E_n \rightarrow E_{n_1}$. Repeat the process, we get a cycle $E_n\to E_{n_1} \to \ldots \to E_n$ in the $L$-isogeny graph. We want to prove that the cycle $E_n\to E_{n_1} \to \ldots \to E_n$ is simple, which means that every vertices in this cycle appears only once. Rewrite the isogenies as following:
\begin{equation}\label{e1}
   \xymatrix{&E_n \ar[r]^{\hat{\phi}_n} &E \ar[r]^{[a+ b\sqrt{-d}]} &E \ar[r]^{[\ell]} &E \ar[r]^{[a+ b\sqrt{-d}]} &\ldots \ar[r] &E \ar[r]^{[a+ b\sqrt{-d}]} &E \ar[r]^{\phi_{n}} &E_n}.
\end{equation} \par
First, we claim that different $G_n$'s generate non-isomorphic elliptic curves, so the isogeny $E_n \to E_{n_1} \to \ldots $ can return to $E_n$ if and only if there exist a positive integer $k$ such that $[a+b\sqrt{-d}]^k(G_n)=G_n$. \par
Let $I_{s}$ and $I_{t}$ be the kernel ideals of $\phi_{s}:E\to E_{s}$ and $\phi_{t}:E\to E_{t}$ with $s,t \in \{0,1,\ldots, \ell \}$. We recall that $E_{s}$ and $E_{t}$ are isomorphic if and only if there exists $\mu \in B_{p,\infty}^{*}$ such that $I_{s} \mu =I_{t}$ with $\text{Nrd}(\mu)=1$ and $\mu \neq \pm 1$. Moreover, $\ell \in I_{s}$, so $\ell \mu \in I_{t} \subseteq \mathcal{O}$. If $\ell$ does not split in $\mathbb{Z}[\tau]$, then there exists an element $\beta \in \mathcal{O}$ with $\text{Nrd}(\beta)=\ell^2$ but $\beta \notin \mathbb{Z}[\tau]$. Then $\mathbb{Z}[\beta]$ is an imaginary quadratic order which can be embedded in $\mathcal{O}$, and the absolute value $D''$ of the discriminant of $\mathbb{Z}[\beta]$ satisfies $D'' \leq 4\text{Nrd}(\beta)=4\ell^2$. By Theorem 2 in \cite{Kan}, we have $4p<4\ell^2 D$ since $\mathbb{Z}[\tau]$ and $\mathbb{Z}[\beta]$ are embedded in $\mathcal{O}$. If $p>\ell^2 L D >\ell^2 D$, then such $\mu$ does not exist and $E_{s}$ is not isomorphic to $E_{t}$. Moreover, we have proved that different $G_n$'s generate non-isomorphic elliptic curves.\par
Next, we prove $[s+t\sqrt{-d}](G_n)=G_n$ if and only if $\ell | t$ for $s, t \in \mathbb{Z}$, so $m$ is the smallest positive integer such that $[a+b\sqrt{-d}]^m(G_n)=G_n$. If $\ell \mid t$, then $[s+t\sqrt{-d}](G_n)=G_n$. On the contrary, if $[s+t\sqrt{-d}](G_n)=G_n$, then $[s+t\sqrt{-d}]P \in G_n$ for any $P\in G_n$. We have $[t\sqrt{-d}]P\in G_n$ for any $P\in G_n$, then $\ell \mid t$ and $[t\sqrt{-d}]P=\infty$. If not, we have $[\sqrt{-d}]P \in G_n$ and $[a+b\sqrt{-d}](G_n) \subseteq G_n$ which is a contradiction.\par
We have proved that the cycle $E_n\to E_{n_1} \to \ldots \to E_n$ is an $m$-cycle. Moreover, since $[a-b\sqrt{-d}]^m=x-y\sqrt{-d}$, there is another $m$-cycle at $j(E_n)$ in $\mathcal{G}_L(\overline{\mathbb{F}}_p,\ell+1)$.\par
If $\tau=\frac{1+\sqrt{-d}}{2}$, the proof is similar.\par
\end{proof}
\begin{remark}
Assume $\ell$ splits in $\mathbb{Z}[\tau]$, we can discuss whether $I_{s}$ and $I_{t}$ are in the same class as in \cite{MR4019136} if we know the endomorphism ring of $E$. In general, if $p>\ell^2 L D$ and $\ell$ splits in $\mathbb{Z}[\tau]$, we can construct two cycles with lengths $m$ at $j(E_n)$ in $\mathcal{G}_L(\overline{\mathbb{F}}_p,\ell+1)$ without backtracking but they may not be simple.
\end{remark}
The following example shows that the conditions in Theorem 2 are not necessary.
\begin{example}
Because $j_0, \ldots, j_5$ are different in Example 1, the conclusion of Theorem 2 also holds even if $p=3461<5^2 \times 28L$. Since the class number of $\mathbb{Z}[\sqrt{-7}]$ is one, by Deuring's reducing and lifting theorems, $\mathbb{Z}[5\sqrt{-7}]$ is optimally embedded in $\mathcal{O}_n \cong \text{End}(E(j_n))$ for $n=0,\ldots,5$. For $11=(2+\sqrt{-7})(2-\sqrt{-7})$, we compute $m=3$,  so there exist $3$-cycles at $j_n$ in $\mathcal{G}_{11}(\overline{\mathbb{F}}_p,\ell+1)$ by Theorem 2. For $L=23=4^2+7\cdot 1^2$, we compute $m=6$, so there exist $6$-cycles at $j_n$ in  $\mathcal{G}_{23}(\overline{\mathbb{F}}_p,\ell+1)$ by Theorem 2.\par
\end{example}
In the following of this section, we will deal with the special cases when $j(E)=1728$ or $0$. Let us recall a result in \cite{MR4019136}:\\
\begin{lemma}
  Suppose $\ell >3$.\\
(1) If $p \equiv  3 \ \text{mod} \ 4$ and $p > 4\ell^2$, there are $\frac{1}{2}(\ell - (\frac{ -1}{\ell}))$ vertices adjacent to $1728$ in $\mathcal{G}_{\ell}(\overline{\mathbb{F}}_p)$, each connecting $1728$ with 2 edges.\\
(2) If $p \equiv  2 \ \text{mod} \ 3$ and $p > 3\ell ^2$, there are $\frac{1}{3}(\ell - (\frac{\ell}{3}))$ vertices adjacent to $0$ in $\mathcal{G}_{\ell}(\overline{\mathbb{F}}_p)$, each connecting $0$ with 3 edges.
\end{lemma}
For $j(E)=1728$, first, we suppose $\ell>3$. If $p \equiv  \ 3 \ \text{mod} \ 4$ and $p > 4\ell^2$, by Lemma 2, we can label the vertices adjacent to $1728$ in $\mathcal{G}_{\ell}(\overline{\mathbb{F}}_p)$ with $j_n$ for $n=1, \ldots, \frac{1}{2}(\ell - (\frac{ -1}{\ell}))$ and denote $E_n=E(j_n)$. We know $\mathbb{Z}[i]$ is optimally embedded in $\text{End}(E(1728))$.
\begin{theorem}
Let $\ell$ be an odd prime and $j_1,\ldots, j_{\frac{1}{2}(\ell - (\frac{ -1}{\ell}))}$ be the vertices adjacent to $1728$ in $\mathcal{G}_{\ell}(\overline{\mathbb{F}}_p)$. If $p \equiv  \ 3 \ \text{mod} \ 4$ and $p > 4\ell^2 L$, then there exist two $m$-cycles at every $j_n$ in $\mathcal{G}_{L}(\overline{\mathbb{F}}_p)$ for $L \equiv 1\ \text{mod} \ 4$.
\end{theorem}
\begin{proof}
Suppose $\ell>3$. If $p \equiv  \ 3 \ \text{mod} \ 4$ and $p > 4\ell^2 L$, then $\mathbb{Z}[\ell i]$ is optimally embedded in $\text{End}(E(j_n))$ and $I_{n,1}$ and $I_{n,2}$ are not principal for every $n\in \{1, \ldots,\frac{1}{2}(\ell - (\frac{ -1}{\ell})) \}$. If $L=(a+bi)(a-bi)$, then $m$ is the smallest positive integer such that $(a+bi)^m=x+yi$ with $\ell \mid y$ or $\ell \mid x$. Let $G'_n = [i](G_n)$, then $G'_n$ is the kernel of $\hat{\phi}_n \circ [i]:E(1728) \to E(1728) \to E(j_n)$. $\hat{\phi}_n$ and $\hat{\phi}_n \circ [i]$ are the two $\ell$-isogenies between $E(1728)$ and $E(j_n)$. For $s,t \in \mathbb{Z}$, we have that $[s+ti](G_n)=G_n$ (resp. $G'_n$) if and only if $\ell \mid t$ (resp. $\ell \mid s$). As in the proof of Theorem 2, there exist two $m$-cycles at $j_n$ in the supersingular isogeny graphs $\mathcal{G}_{L}(\overline{\mathbb{F}}_p)$.\par
For $\ell =3$, we have $\Phi_3(X, 1728)=(X^2-153542016X-1790957481984)^2$. If $p \equiv  \ 3 \ \text{mod} \ 4$ and $p> 31$, then there are two vertices $j_1$ and $j_2$ adjacent to $1728$ in $\mathcal{G}_{\ell}(\overline{\mathbb{F}}_p)$, each connecting $1728$ with 2 edges. For $L=a^2+b^2$, if $3\mid a$ or $3\mid b$, there are two loops at $j_1$ and $j_2$ in $\mathcal{G}_{L}(\overline{\mathbb{F}}_p)$. If $3\nmid ab$, we have $3\mid(a^2-b^2)$. There are two $2$-cycles at $j_1$ and $j_2$ in $\mathcal{G}_{L}(\overline{\mathbb{F}}_p)$.
\end{proof}
\begin{remark}
  For $\ell =2$, we have $\Phi_2(X,1728)=(X-1728)(X-66^3)^2$. If $p \equiv  \ 3 \ \text{mod} \ 4$ and $p> 11$, then $66^3$ is a supersingular $j$-invariant which is different from $1728$. $\mathbb{Z}[2i]$ is optimally embedded in $\mathcal{O}(66^3)$. For $L=a^2+b^2$, there are at least two loops at $66^3$ in $\mathcal{G}_{L}(\overline{\mathbb{F}}_p)$.
\end{remark}
For $j(E)=0$ and $\ell>3$, if $p \equiv \ 2 \ \text{mod} \ 3$ and $p > 3\ell ^2$, by Lemma 2, we can label the vertices adjacent to $0$ in $\mathcal{G}_{\ell}(\overline{\mathbb{F}}_p)$ with $j_n$ for $n=1, \ldots, \frac{1}{3}(\ell - (\frac{\ell}{3}))$. Let $\epsilon = \frac{1+\sqrt{-3}}{2}$, we have $\mathbb{Z}[\epsilon]$ is optimally embedded in $\text{End}(E(0))$.
\begin{theorem}
Let $\ell$ be an odd prime and $j_1,\ldots, j_{\frac{1}{3}(\ell - (\frac{\ell}{3}))}$ be the vertices adjacent to $0$ in $\mathcal{G}_{\ell}(\overline{\mathbb{F}}_p)$. If $p \equiv \ 2 \ \text{mod} \ 3$ and $p > 3\ell ^2L$, then there exist two $m$-cycles at every $j_n$ in $\mathcal{G}_{L}(\overline{\mathbb{F}}_p)$ for $L\equiv 1 \ \text{mod} \ 3$.
\end{theorem}
\begin{proof}
For $\ell >3$, the proof is similar to that of Theorem 3.\par
For $\ell =3$, we have $\Phi_3(X, 0)=X(X-12288000)^3$. If $p \equiv 2 \ \text{mod} \ 3$ and $p> 23$, we have $-12288000$ is a supersingular $j$-invariant which is different from $0$. $\mathbb{Z}[3\epsilon]$ is optimally embedded in $\mathcal{O}(-12288000)$. For $L \equiv \ 1 \ \text{mod} \ 3$, there are at least two loops at $-12288000$ in $\mathcal{G}_{L}(\overline{\mathbb{F}}_p)$.\par
\end{proof}
\begin{remark}
For $\ell =2$, we have $\Phi_2(X,0)=(X-54000)^3$. If $p \equiv 2 \ \text{mod} \ 3$ and $p> 11$, then $54000$ is a supersingular $j$-invariant which is different from $0$. It is easy to show that $\mathbb{Z}[\sqrt{-3}]$ is optimally embedded in $\mathcal{O}(54000)$. For $L \equiv \ 1 \ \text{mod} \ 3$, there are at least two loops at $54000$ in $\mathcal{G}_{L}(\overline{\mathbb{F}}_p)$.
\end{remark}
As we can see, $m$ plays an important role in our theorems. Denote $O=\mathbb{Z}[\tau]$ and $O'=\mathbb{Z}[\ell \tau]$. Let $h$ and $h'$ be the class number of $O$ and $O'$ respectively. We have the following formula in Chapter 7 of \cite{MR3236783}
$$\frac{h'}{h}=\frac{\ell}{[O^*:O'^*]}\bigg(1-\bigg(\frac{O}{\ell}\bigg)\frac{1}{\ell}\bigg),$$
where $O^*$ and $O'^*$ are the unit groups of $O$ and $O'$ respectively.\par
It is easy to see that $m \mid \frac{h'}{h}$. The following example shows that $m$ can take any possible value.
\begin{example}
  Let us return to Example 1. For $\tau=\sqrt{-7}$ and $\ell =5$, we know that $5$ is inert in $\mathbb{Z}[\sqrt{-7}]$ and $\frac{h'}{h}=6$. We have $m=3$ or $6$ if $L=11$ or $23$. Furthermore, when $L=179$ or $53$, we have $m=1$ or $2$ respectively.\par

\end{example}

\section{$2$-Cycles}
As in Section 4, suppose that $\mathbb{Z}[\tau]$ is optimally embedded in $\mathcal{O}\cong\text{End}(E)$. If $E_n$ and $E$ are $\ell$-isogenous, we can get the sufficient conditions under which there are $2$-cycles at $j(E_n)$ in $\mathcal{G}_L(\overline{\mathbb{F}}_p,\ell+1)$.
\begin{corollary}
Suppose that $\mathbb{Z}[\tau]$ and $\mathbb{Z}[\ell \tau]$ are optimally embedded in $\text{End}(E)$ and $\text{End}(E_n)$ respectively and $p>\ell^2 L D$ where $D$ is the absolute value of the discriminant of $\mathbb{Z}[\tau]$.\par
(1)Suppose $p \equiv  3 \ \text{mod} \ 4$. If $\tau=i$ and $\ell >2$, then there exist $2$-cycles at $j(E_n)$ in $\mathcal{G}_L(\overline{\mathbb{F}}_p,\ell+1)$ if $L=a^2+b^2$ with $\ell \mid (a^2-b^2)$ and $\ell \nmid a$.\par
(2)Suppose $p \equiv  2 \ \text{mod} \ 3$. If $\tau=\epsilon$ and $\ell >3$, then there exist $2$-cycles at $j(E_n)$ in $\mathcal{G}_L(\overline{\mathbb{F}}_p,\ell+1)$ if $L=a^2+3b^2$ with $\ell \nmid a$, $\ell \nmid b$, $\ell \nmid (a+b)$ and $\ell \mid (a^2-b^2)$ (or $\ell \mid (b^2+2ab)$, or $\ell \mid (a^2+2ab)$).\par
(3)If $\mathbb{\tau}\neq \mathbb{Z}[i]$, $\mathbb{Z}[\epsilon]$ and $\ell>2$, then there exist $2$-cycles at $j(E_n)$ in $\mathcal{G}_L(\overline{\mathbb{F}}_p,\ell+1)$ if $L=(a+b\tau)(a-b\tau)$ with $\ell \mid a$ and $\ell \nmid b$.\par
(4)If $\tau =\sqrt{-d}\neq i$ and $\ell=2$, then there exist $2$-cycles at $j(E_n)$ in $\mathcal{G}_L(\overline{\mathbb{F}}_p,\ell+1)$ if $L=(a+b\sqrt{-d})(a-b\sqrt{-d})$ with $2 \nmid b$.\par
\end{corollary}
For $m=1$, there are loops at $j(E_n)$ in $\mathcal{G}_L(\overline{\mathbb{F}}_p)$. The method in \cite{MR3872948} can be used to determine the upper bounds on $p$ for which $j(E_n)$ has unexpected loops in $\mathcal{G}_L(\overline{\mathbb{F}}_p)$. If $E$, $E_n$ and $L=(a+b\tau)(a+b\bar{\tau})$ satisfy the conditions in Corollary 5.1, we denote the target elliptic curve of the two $L$-isogenies $\psi_{n,\pm}$ from $E_n$ by $E_n'$. In the remainder of this section, we will determine an upper bound on $p$ for which there exist unexpected isogenies from $E_n$ to $E_n'$ of degree $L$.\par
Since $\mathbb{Z}[\ell \tau]$ is optimally embedded in $\mathcal{O}_n$ and $L=(a+b\tau)(a+b\bar{\tau})$, we have that $I_{n,1}=(L,\ell(a+b\tau))$ and $I_{n,2}=(L,\ell(a+b\bar{\tau}))$ are the kernel ideals corresponding to $\psi_{n,+}$ and $\psi_{n,-}$. Moreover, we know that $I_{n,1}$ and $I_{n,2}$ are in the same ideal class if $L$ satisfies the conditions in Corollary 5.1. Let $D$ denote the absolute value of the discriminant of $\mathbb{Z}[\tau]$.
\begin{theorem}
  Suppose that $\ell$ is ramified or inert in $\mathbb{Z}[\tau]$ and $L$ satisfies the condition in Corollary 5.1. If $p>D\ell^2L$, then there are only two $L$-isogenies from $E_n$ to $E_n'$.
\end{theorem}
\begin{proof}
    If there is another $L$-isogeny from $E_n$ to $E_n'$, the corresponding kernel ideal $J$ must belong to $X_L$ and $I_{n,1}=J\mu$, where $\mu\in B_{p,\infty}^*$ with $\text{Nrd}(\mu)=1$. Since $L \in J$, we have $L \mu \in I_{n,1}$ and $\mu \in L^{-1}I_{n,1}$. There exist $x,y \in \mathcal{O}_n$ such that $\mu=L^{-1}(xL+y\ell(a+b\tau))=L^{-1}(x(a+b\bar{\tau})+\ell
    y)(a+b\tau)$, and $\alpha=x(a+b\bar{\tau})+\ell y$ with $\text{Nrd}(\alpha)=L$. We have $\ell \alpha\in \mathcal{O}$ since $\ell x$, $\ell y$ and $a+b\bar{\tau}$ are in $\mathcal{O}$.\par
    If $\ell \alpha \in \mathbb{Z}[\tau]$, then $\alpha \in \ell^{-1}\mathbb{Z}[\tau]$. Since $\ell$ is ramified or inert in $\mathbb{Z}[\tau]$, the set of elements with norm $\ell^2$ in $\mathbb{Z}[\tau]$ is $\{\varepsilon \ell: \varepsilon \in \mathbb{Z}[\tau]^* \}$. For simplicity, we can assume $\alpha =a+ b\tau$ or $a+b\bar{\tau}$. If $\alpha=a+b\bar{\tau}$, then $\mu=1$ and $J=I_{n,1}$. If $\alpha= a+b\tau$, then $\mu=  \frac{a+b\tau}{a+b\bar{\tau}}$ and $J=I_{n,2}$. If $\ell \alpha$ is not in $\mathbb{Z}[\tau]$, by Theorem 2 in \cite{Kan}, we have $4p\leq 4D\text{Nrd}(\ell\alpha)=4D\ell^2L$ since $\mathbb{Z}[\tau]$ and $\mathbb{Z}[\ell \alpha]$ are embedded in $\mathcal{O}$. We assume $p> D\ell^2L$, so such $\alpha$ does not exist. This proves the theorem.
\end{proof}
The following examples show that the bound in Theorem 5 is sharp.
\begin{example}
If $p\equiv  3 \ \text{mod} \ 4$ and $\ell=2$, then $j(2i)=66^3$ is a supersingular $j$-invariant in $\mathbb{F}_p$ and $2$ is ramified in $\mathbb{Z}[2i]$. For $L=13=(3+2i)(3-2i)$ and $p=827$, $j_1=774\beta+169$, $j_2=53\beta+169$ and $j_3=1728$ are the vertices adjacent to $66^3$ in $\mathcal{G}_{2}(\overline{\mathbb{F}}_p)$ where $\beta^2+1=0$. There exist three $13$-isogenies from $j_1$ to $j_2$. In fact, $p=827$ is the largest prime satisfying $p\equiv \ 3 \ \text{mod} \ 4$ and $p<16 \times 2^2 \times 13=832$.\par
If $\left ( \frac{-7}{p} \right )=-1$ and $\ell=3$, then $j(\frac{1+\sqrt{-7}}{2})=-15^3$ is a supersingular $j$-invariant in $\mathbb{F}_p$ and $3$ is inert in $\mathbb{Z}[\frac{1+\sqrt{-7}}{2}]$. For $L=37=(3+2\sqrt{-7})(3-2\sqrt{-7})$ and $p=2309$, $j_1=860\beta+1506$ and $j_2=1449\beta+1506$ are two vertices adjacent to $-15^3$ in $\mathcal{G}_{3}(\overline{\mathbb{F}}_p)$ where $\beta^2+\beta+1=0$. There exist three $37$-isogenies from $j_1$ to $j_2$. In fact, $p=2309$ is the largest prime satisfying $\left ( \frac{-7}{p} \right )=-1$ and $p<7 \times 3^2 \times 37=2331$.
\end{example}

\section{Conclusion}

For a supersingular elliptic curve $E$, if an imaginary quadratic order $\mathbb{Z}[\tau]$ can be embedded in $\text{End}(E)$ and a prime $L$ splits into two principal ideals in $\mathbb{Z}[\tau]$, we construct loops or cycles in the supersingular $L$-isogeny graph at the vertices $\{ j(E_n) \}_{n=0,\ldots,\ell}$ which are neighbors of $j(E)$ in the $\ell$-isogeny graph, where $\ell$ is a prime different from $L$. If $\mathbb{Z}[\ell \tau]$ is optimally embedded in $\text{End}(E_n)$ and $L\mathbb{Z}[\ell \tau]=\mathfrak{L}\mathfrak{L}'$, then the length of each cycle which we construct at $j(E_n)$ is the order of $\mathfrak{L}$ in the class group of $\mathbb{Z}[\ell \tau]$ essentially. \par
If we walk two steps from $E$ in the supersingular $\ell$-isogeny graph, we can get $\ell(\ell+1)$ vertices in general and construct loops and cycles at these vertices in the $L$-isogeny graph by our method. In general, similar results hold for any number of steps. These results show a deeper connection between different supersingular isogeny graphs.\par

\small
\section*{}
\bibliographystyle{plain}

\bibliography{reference}

\end{document}